\documentclass[a4paper,10pt,reqno]{amsart}
\usepackage{enumerate}
\usepackage{amsfonts}
\usepackage{color}
\usepackage{amsmath}
\usepackage{amssymb}
\usepackage[latin1]{inputenc}
\newcommand{\Z}{{\mathbb Z}}
\newcommand{\RR}{{\mathbb R}}
\newcommand{\R}{{\mathbb R}}
\newcommand{\CC}{{\mathbb C}}
\newcommand{\N}{{\mathbb N}}

\newcommand{\qn}{Q^{ (N)}}
\newcommand{\qd}{Q^{(D)}}
\newcommand{\nl}{L^{ (N)}}
\newcommand{\dl}{L^{ (D)}}
\newcommand{\wl}{\widetilde{L}}

\newcommand{\ph}{{\varphi}}

\newcommand{\lm}{{\lambda}}

\newcommand{\arccosh}{{\mathrm {arccosh}}}

\newcommand{\from}{\colon}

\renewcommand\ge{\geqslant}

\newtheorem{lemma}{Lemma}[section]
\newtheorem{coro}[lemma]{Corollary}
\newtheorem{definition}[lemma]{Definition}
\newtheorem{prop}[lemma]{Proposition}

\newtheorem{theorem}[lemma]{Theorem}

 \newlength\headseptemp

\newcommand{\Condition}{(\mathcal{C})}

\newcommand{\Hmm}[1]{\leavevmode{\marginpar{\tiny%
$\hbox to 0mm{\hspace*{-0.5mm}$\leftarrow$\hss}%
\vcenter{\vrule depth 0.1mm height 0.1mm width \the\marginparwidth}%
\hbox to 0mm{\hss$\rightarrow$\hspace*{-0.5mm}}$\\\relax\raggedright #1}}}
\begin{document}
\title[]{Laplacians on infinite  graphs: Dirichlet and Neumann boundary conditions}

\author[]{Sebastian Haeseler$^1$}
\author[]{Matthias Keller$^2$}
\author[]{Daniel Lenz$^3$}
\author[]{Rados{\l}aw Wojciechowski$^4$}
\address{$^1$ Mathematisches Institut, Friedrich Schiller Universit\"at Jena,
  D- 07743 Jena, Germany, sebastian.haeseler@uni-jena.de.}
\address{$^2$ Mathematisches Institut, Friedrich Schiller Universit\"at Jena,
  D-  07743 Jena, Germany, m.keller@uni-jena.de.}
\address{ $^3$ Mathematisches Institut, Friedrich Schiller Universit\"at Jena,
  D- 07743 Jena, Germany, daniel.lenz@uni-jena.de,\\
  URL: http://www.analysis-lenz.uni-jena.de/ }
\address{$^4$ Department of Mathematics and Computer Science,
York College of The City University of New York, Jamaica, NY 11451, USA, e-mail: rwojciechowski@gc.cuny.edu
}

\begin{abstract}  We study  Laplacians associated to a graph and single out a class of such operators with special regularity properties.  In the case of locally finite graphs, this class consists of  all selfadjoint,  non-negative  restrictions of the standard formal  Laplacian and we can characterize the Dirichlet and Neumann Laplacians as the largest and smallest Markovian restrictions of the standard formal Laplacian.  In the case of general graphs, this class contains  the  Dirichlet and Neumann Laplacians  and we describe how these may differ from each other, characterize when they agree, and study connections  to essential selfadjointness and stochastic completeness.

Finally, we study basic common features of all Laplacians associated to a graph.  In particular, we characterize when the associated semigroup is positivity improving and present some basic estimates on its long term behavior. We also discuss some situations in which the Laplacian associated to a graph is unique and, in this context,  characterize its boundedness.
\end{abstract}
\date{\today} %
\maketitle
\begin{center}
\emph{}
\end{center}

\section{Introduction}
Laplacians on graphs have been studied for a long time  (see, e.g., the monographs \cite{Chu,Col} and references therein).  Much of the research has been devoted to finite graphs and bounded Laplacians.

After sporadic earlier investigations, notably by Dodziuk \cite{Dod0} and Mohar \cite{Mo}, certain  properties related to unboundedness  of the associated Laplacians on infinite graphs  have become a focus of attention in recent years. This concerns, in particular,  essential selfadjointness \cite{dVTHT, Dod, Gol, GS, Jor, JP, KL1, Ma, TH, Web, Woj1}, stochastic (in)completeness \cite{Dod, DM, GHM, Hua, KL1,KL2,  Web, Woj1,Woj2, Woj3} and suitable isoperimetric inequalities \cite{CGY,Dod, Fuj,Kel,KL2,KP, Woj1,Woj2} (see references in the cited works for  further literature as well).

It turns out that all of these works deal with what could be called the `Dirichlet Laplacian' on a graph.  In the essentially selfadjoint case, of course, this is the only Laplacian.  In general, however, further selfadjoint Laplacians exist.  In particular, there exists a `Neumann Laplacian.'  It is not clear when the two Laplacians agree and which properties they share (if they do not agree).  This is the starting point of this paper.  More generally, our aim is to investigate the following three related questions:

 \begin{itemize}
 \item[(Q-1)] Which operators can be considered to be Laplacians associated to a graph?
 \item[(Q-2)] How are these operators related and what are the differences between them?
 \item[(Q-3)] What are the basic properties common to all of them?
 \end{itemize}

We now provide a general overview of the paper and our results on these questions.  For precise statements and definitions of the quantities involved we refer to later sections.

\smallskip

In Section~\ref{Framework}, we give an exposition of basic notation and concepts.  In particular, we introduce graphs, the standard formal Laplacian associated to a graph, and the forms $\qd$ and $\qn$ giving the Dirichlet and Neumann Laplacians, respectively. We also prove a result showing that the `weak domain' of definition of the formal Laplacian actually agrees with its domain (Theorem~\ref{regular}). This result is important for our further considerations and may also be of independent interest.

As for (Q-1), which is studied in Section~\ref{Laplacians}, we note that any graph comes with both a standard formal operator $\wl$ and a closed form $\qd$.  In some sense, $\wl$ is the `maximal' Laplacian associated to the graph and $\qd$ is the `largest' closed form associated to $\wl$.  This leads us to single out Laplacians and forms associated to a graph which satisfy a regularity-type condition, called $\Condition$, implying that the form lies between $\qd$ and $\wl$.  A precise concept is given in Definition~\ref{Associated}.

In the case of locally finite graphs, the corresponding Laplacians turn out to be exactly the selfadjoint restrictions of $\wl$ which are bounded below (Theorem~\ref{locally-finite}).  In the case of general graphs, we do not have an explicit  description of all Laplacians satisfying $\Condition$ in terms of $\wl$.  However, we can show that the Dirichlet operator and the Neumann operator (and all operators between them in the sense of forms) satisfy this condition (Proposition~\ref{PI-prop}). In this sense, our framework seems to be sufficient to address questions (Q-2) and (Q-3) and, in particular, to study the Dirichlet and Neumann Laplacians.

As for (Q-2), our framework allows us to obtain, in an easy way, a general description of how a form satisfying $\Condition$ can be seen as an extension of $\qd$.  This is given in Theorem~\ref{general} of Section~\ref{Laplacians}. This theorem can be seen as a form-type analogue of some basic results in von Neumann extension theory.  On a technical level, the main topic is the description of $1$-harmonic functions $u$ in the domain of the form $Q$ associated to the graph, i.e., $u$ with $$(\wl + 1) u =0$$
belonging to the space $D(Q)$.

\smallskip

In Sections~\ref{Forms} and \ref{Neumann}, we then have a closer look at (Q-2) for Dirichlet and Neumann Laplacians:
In Section~\ref{Forms}, we describe the `difference' between Dirichlet and Neumann Laplacians if they do not agree (Theorem~\ref{main-qd-qn}) and give a  characterization of when the Dirichlet and Neumann Laplacians agree (Corollary~\ref{char-qd-qn}).  We also discuss how our results are related to recent work of Colin de Verdi\`{e}re, Torki-Hamza and Truc \cite{dVTHT,TH}.  In fact, while somewhat different in spirit, our description of the difference between Neumann and Dirichlet Laplacians in Theorem~\ref{main-qd-qn} is certainly inspired by \cite{dVTHT}.

\smallskip

We then turn to characterizing Dirichlet and Neumann Laplacians in the framework of Laplacians associated to a graph in Section~\ref{Neumann}.  Our approach gives immediately that the Dirichlet Laplacian is, in a precise sense, the largest Laplacian associated to a graph.  The main thrust of Section~\ref{Neumann} is to show that the Neumann Laplacian is the smallest Laplacian associated to a graph within the class of Markovian operators (i.e., operators associated to a Dirichlet form).  For our results to work, we have to make the additional assumption of local finiteness of the graph.  For locally finite graphs, Theorem~\ref{char-Neumann} then gives that, among the Markovian restrictions of $\wl$, the Dirichlet Laplacian is the biggest and the Neumann Laplacian is the smallest. While similar results are known for the usual Laplacians on subsets of Euclidean space \cite{Fuk}, we are not aware of any earlier result of this type for graphs.

\smallskip

It is remarkable that the agreement of $\qd$ and $\qn$ is equivalent to the triviality of solutions to $(\wl + 1) u =0$ in $D(\qn)$, as the solvability of this equation in other spaces is known to be related to stochastic completeness and to essential selfadjointness.  In this way, essential selfadjointness, stochastic completeness, and uniqueness of the operator are related.  Details are discussed in Section~\ref{Equation}.  In particular, by examples we show that, apart from the `obvious' implications, no implications between these concepts hold in general.  More specifically, we show that stochastic completeness and essential selfadjointness are not related in general.

\smallskip

We finally turn to question (Q-3) and discuss basics of a theory valid for both Neumann and Dirichlet Laplacians (and many others) in Sections~\ref{Maximum} and \ref{Analogue}.  There, we are mostly concerned with the semigroup associated to these operators:

First, we present a maximum principle for solutions of $(\wl +1) u =0$ and use it to characterize when the semigroup is positivity improving in Section~\ref{Maximum}. This generalizes the corresponding considerations for the Dirichlet Laplacian in \cite{KL1} (see \cite{Web, Woj1, Dav3} for earlier treatment of special Dirichlet Laplacians as well).

We then discuss an analogue to a result of Li on Laplacians on manifolds in our context in Section~\ref{Analogue}.  This result has already been obtained recently in a rather general context \cite{KLVW}.  Here, we present a different proof which is adapted to the graph case.

\smallskip

In Section \ref{Boundedness}, 
we conclude the paper with a study of   situations in which there is  only one selfadjoint restriction of $\wl$. This complements and completes the considerations of the earlier sections which deal with the differences and common features of all selfadjoint restrictions.


\smallskip

In some sense, this paper can be seen as a complement to \cite{KL1}.  There, basic features of the Dirichlet Laplacian were discussed.  Here, we focus on the general case.

\section{Framework and basic results}\label{Framework}
Throughout the paper, let $V$ be a finite or countably infinite set and $m$ a measure on $V$ with full support (i.e., $m$ is a map on $V$ taking values in $(0,\infty)$). We then call $(V,m)$ a \textit{discrete measure space}.  The set of all function from $V$ to $\CC$ is denoted by $C(V)$.

We will introduce operators on $\ell^2 (V,m)$ using Dirichlet forms. To do so, we first briefly recall a few standard facts on forms (see, e.g., \cite{Dav3,Fuk}).    Some of the standard literature on Dirichlet forms only deals with real Hilbert spaces. However, this can easily be extended to complex  Hilbert spaces. Some details are discussed in Appendix \ref{real}.
A \textit{form} $Q$ on a (complex)  Hilbert space  with domain of definition given by a dense subspace $D(Q)$ is a sesquilinear map $Q : D(Q)\times D(Q)\longrightarrow \CC$. The form $Q$ is called \textit{non-negative} if $Q(u,u)\geq 0$ for all $u\in D(Q)$ and \textit{symmetric} if $Q(u,v) = \overline{Q(v,u)}$ for all $u,v\in D(Q)$.  A non-negative symmetric form $Q$ is called \textit{closed} if $D(Q)$ with the inner product
$$\langle u,v\rangle_Q := Q(u,v)+ \langle u, v\rangle$$
is complete, i.e., a Hilbert space.  To each such  form there exists  a unique selfadjoint operator $L$ with
$$D(Q) = \mbox{Domain of definition of $L^{1/2}$}$$
and
$$Q (u,v) = \langle L^{1/2} u, L^{1/2} v\rangle.$$
A  map $C:   \CC\longrightarrow \CC $ with $C(0) =0$ and
$|C(x) - C(y)|\leq |x - y|$ is  called a \textit{normal contraction}.
A closed form $Q$ on a Hilbert space of square integrable functions  is called a \textit{Dirichlet form}  if
$$Q (Cu,C u) \leq Q(u,u)$$
for all $u\in D(Q)$ and all normal contractions $C$.
The relevance of Dirichlet forms comes from the fact that the associated semigroups $(e^{-tL})_{t\geq 0}$ and resolvents $\alpha (L + \alpha)^{-1}$, $\alpha >0$, are \textit{positivity preserving},  i.e., map non-negative functions to non-negative functions and provide  contractions on the space of bounded functions (see, e.g., \cite{BH,Dav2}).

\smallskip

After this summary on forms, we now come to a discussion of graphs over $(V,m)$ and the associated operators.  To a large extent we follow \cite{KL1,HK} to which we refer for
further details and proofs not given below.  (Note that our notation deviates from the notations of \cite{KL1,HK} - which are only concerned with the Dirichlet Laplacian -  in the following way: We denote by $\qn$ the form denoted by $Q^{max}$  in \cite{KL1} and  by $\qd$ the form denoted by $Q$ in \cite{KL1}.)

By a \textit{symmetric weighted graph} over $V$ we mean a pair $(b,c)$ consisting of a map $c \from V \to [0,\infty)$ and a map
 $b \from V\times V \to [0,\infty)$
satisfying the following  properties:
\begin{itemize}
\item $b(x,x)=0$ for all $x\in V$
\item $b(x,y)= b(y,x)$ for all $x,y\in V$
\item $\sum_{y\in V} b(x,y) <\infty$ for all $x\in V$.
\end{itemize}
Then, $x,y\in V$ with $b(x,y)>0$ are called \textit{neighbors} and thought to be
connected by an edge with weight $b(x,y)$.  More generally, $x,y\in V$ are called \textit{connected} by the \textit{path} $(x_0,x_1,\ldots, x_{n+1})$  if $x_0,x_1,\ldots, x_{n+1} \in V$ satisfy $b(x_i, x_{i+1}) >0$,
$i=0,\ldots, n$, with $x_0 = x$ and $x_{n+1} = y$.  A \textit{connected component} of the graph is a maximal subset of $V$ such that all elements in this set are connected. If $V$ has only one connected component, i.e., if any two $x,y\in V$ are connected, then $(b,c)$ is called \textit{connected}.  Symmetric weighted graphs over $(V,m)$ are also known as symmetric Markov chains over $(V,m)$.

\smallskip

We are now going to associate forms and operators to each graph $(b,c)$ over $(V,m)$. These forms and operators  will, of course, depend on the choice of $(b,c,m)$. We will mostly omit this dependence on $(b,c,m)$ in our notation  and only add the corresponding subscripts when necessary to avoid confusion.

To the graph $(b,c)$ over $(V,m)$  we  associate the form  $\qn$  on the Hilbert space $\ell^2 (V,m)$ with domain of definition $D(\qn)$ given by the subspace
$$D (\qn) := \{ u\in \ell^2 (V,m) : \frac{1}{2}\sum_{x,y\in V} b(x,y) |u(x) - u(y)|^2 + \sum_{x\in V} c(x) |u(x)|^2\ < \infty \}$$
and the map
$$\qn : D(\qn) \times D(\qn) \longrightarrow \CC $$
by
$$\qn (u,v) := \frac{1}{2}\sum_{x,y\in V} b(x,y) (u(x) - u(y)) \overline{(v(x) - v(y))  } + \sum_{x\in V} c(x) u(x) \overline{v(x)}.$$
Then, $\qn$ is symmetric, non-negative and   closed. The associated operator will be denoted by $\nl$. We can think of $\nl$ as a Laplacian with Neumann-type boundary conditions.

We will be concerned not only with $\qn$ but with further forms as well. In this context, we use the notation $Q_1 \subseteq Q_2$ to mean that
$$D(Q_1)\subseteq D(Q_2)\;\: \mbox{and}\;\: Q_1 (u,v) = Q_2 (u,v)$$
for all $u,v\in D(Q_1).$
Similarly, for non-negative forms $Q_1$ and $Q_2$, we use the notation $Q_1 \leq Q_2$ to mean that
$$ D (Q_2)\subseteq D(Q_1)\;\:\mbox{and}\;\: Q_1 (u,u) \leq Q_2 (u,u)$$
for all $u\in D(Q_2)$.

Obviously, the set $C_c (V)$ of functions from $C (V)$  with finite support belongs to $D(\qn)$. Thus, we can restrict $\qn$ to this set to obtain the  form $Q^{comp}$
$$Q^{comp} \from C_c (V)\times C_c (V)  \longrightarrow \CC,\;\: Q^{comp}(u,v) := \qn(u,v).$$
The form $Q^{comp}$ is not closed but possesses a unique smallest closed extension called the closure and denoted by $\qd$. The associated selfadjoint operator is denoted by  $\dl$. We can think of $\dl$ as a Laplacian with Dirichlet-type boundary conditions.

\medskip

Note that $\qn$ and $\qd$ are Dirichlet forms. This is rather straightforward  to show for $\qn$ and follows for $\qd$ by general principles. By construction, the form $\qd$ is a regular Dirichlet form, viz, $C_c (V)$ is dense in  the form domain with respect to the form norm induced by $\langle \cdot,\cdot\rangle_{\qd}$. In fact, all regular Dirichlet forms on $(V,m)$ are of the form $\qd = \qd_{(b,c)}$ for suitable graphs $(b,c)$ (see, e.g., \cite{KL1}).

\medskip

While the domains of definition of $\dl$ and $\nl$ can, in general, not  be described explicitly, the action of these operators is easily described. To do so, we introduce the \textit{standard formal Laplacian} $\wl$ associated to the graph $(b,c)$ over $(V,m)$.  This operator will be of fundamental importance in all of our considerations.  It is defined on the space
$$\widetilde{F}:= \{ u\in C(V) : \sum_{y\in V} |b(x,y) u(y)|<\infty\;\mbox{for all $x\in V$ } \} $$
 by
$$\wl u (x) :=\frac{1}{m(x)} \sum_{y\in V} b(x,y) (u(x) - u(y)) + \frac{c(x)}{m(x)}  u(x).$$
Note that, for each $x\in V$,  the sum exists  by the assumption that $u$ belongs to $\widetilde{F}$.

\smallskip

It turns out that $\wl $ has a certain regularity property, viz, functions which are weakly in its domain are actually in its domain.  The  crucial identity connecting $\wl$ and the forms we have in mind  is then given by a certain integration by parts. This is discussed next.  We start by introducing the functions which are weakly in the domain of $\wl$ (see \cite{FLW} as well).

\begin{definition}  Let $(V,m)$ be a discrete measure space and $(b,c)$ a graph over $(V,m)$. Then, $\widetilde{F}^\ast$, the weak domain of the formal Laplacian, is defined  by
$$\widetilde{F}^\ast :=\{ u\in C(V): \sum_{x\in V} |u(x) \wl v (x)| m(x)  <\infty\;\:\mbox{for all $v\in C_c (V)$}\}.$$
\end{definition}

Here, comes the first part of the necessary `integration by parts' as shown in \cite{HK} (see \cite{KL1} for related results as well): For $u\in \widetilde{F}$ and $v\in C_c (V)$, the sum
$$\widetilde{Q} (u,v):= \frac{1}{2}\sum_{x,y\in V} b(x,y) (u(x) - u(y))(\overline{v(x) - v(y)}) + \sum_{x\in V} c(x) u(x) \overline{v(x)}$$
converges absolutely and the equality
\begin{eqnarray}\label{PI}
\widetilde{Q} (u,v)  = \sum_{x\in V} \wl u (x) \overline{v(x)} m(x) = \sum_{x\in V} u (x) \overline{\wl v (x)} m(x)
\end{eqnarray}
holds (where all sums are converging absolutely).

After these preparations we can now state a regularity property of $\wl$.

\begin{theorem} \label{regular} Let $(V,m)$ be a discrete measure space and $(b,c)$ a graph over $(V,m)$. Then, $\widetilde{F} = \widetilde{F}^\ast$.
\end{theorem}
\begin{proof} The inclusion $\widetilde{F}\subseteq \widetilde{F}^\ast$ follows from \eqref{PI}.
It remains to show the other inclusion $\widetilde{F}^\ast \subseteq \widetilde{F}$:  Let $u\in \widetilde{F}^\ast$ be given.   We have to show the absolute convergence of $\sum_{z\in V} b(x,z) u(z)$ for any $x\in V$. Let $\delta_x$  be the characteristic function of $\{x \}$.  For each $z\in V$, we set $B_z:=\sum_{y\in V} b(z,y) + c(z)$. Then, a direct calculation shows that
$$ \wl \delta_x (z) = \frac{1}{m(z)} ( B_z \delta_x (z) - b(x,z)).\;\:\; (*)$$
As $\delta_x$ belongs to $C_c (V)$, the absolute convergence of $\sum u(z) \wl \delta_x (z) m(z)$ for each $x\in V$ follows by the assumption on $u$. Now, $(*)$ easily gives the statement.
\end{proof}

\textbf{Remark.} The previous theorem seems particularly remarkable to us as it does not seem to have a direct counterpart in the case of the usual Laplace-Beltrami $\Delta_M$  on a Riemannian manifold $M$.    Certainly, the existence of $\langle u, \Delta v\rangle  $ for all $v \in C_c^\infty (M)$ does not imply any differentiability properties of $u$ (as it will hold, in particular, for any measurable bounded function with compact support).

\bigskip

As a consequence of the previous theorem we obtain that weak generalized eigenfunctions are generalized eigenfunctions:

\begin{coro} \label{eigenfunction}  Let $(V,m)$ be a discrete measure space, $(b,c)$ a graph over $(V,m),$ $u\in C(V)$ and $\lambda\in \RR$. Then, the following assertions are equivalent:
\begin{itemize}
\item[(i)] The function  $u$ belongs to $\widetilde{F}$ and $(\wl- \lambda ) u =0$, i.e., $u$ is a generalized eigenfunction of $\wl$ to the eigenvalue $\lambda$.

\item[(ii)] The function $u$ belongs to $\widetilde{F}^\ast$ and $\sum_{x\in V} u(x) \overline{(\wl - \lambda) v (x)} m(x) =0$ for all $v\in C_c (V)$, i.e., $u$ is a weak generalized eigenfunction of $\wl$.
\end{itemize}
\end{coro}

\section{Laplacians associated to a graph}\label{Laplacians}
In this section, we introduce a special class of operators and forms associated to a graph. As will become clear in the paper, these forms and operators can be considered as particularly regular Laplacians on a graph.  In this section, we develop some basics of their theory.  In particular, Theorem \ref{general} gives a form-type  analogue of what might be seen as a basic ingredient of von Neumann extension theory for symmetric operators.
Moreover, we show that all the `usual' Laplacians fall into our framework. More precisely, we show in Theorem \ref{locally-finite} that, in the locally finite case (and even a bit more generally), our class consists of  the non-negative selfadjoint restrictions of $\wl$. In the case of general graphs, we show that the Dirichlet and Neumann operators (and all operators between them in the sense of forms)  belong to the class.

\bigskip

Whenever we are given a graph with an associated standard formal Laplacian  $\wl$ we call a selfadjoint restriction of $\wl$ a \textit{Laplacian associated to the graph}. If this restriction is bounded below, we call  the induced form a \textit{form  associated to the graph}.  We are going to single out a special class of operators associated to a graph and study some of their properties.  We start with the definition of the class.

\begin{definition}  \label{Associated} (Forms satisfying $\Condition$) Let $(V,m)$ be a discrete measure space, $(b,c)$ a graph over $(V,m)$ and $\wl=\wl_{(b,c,m)}$.   A symmetric form $Q$ on $\ell^2 (V,m)$ with domain $D$  is said to satisfy condition $\Condition$ with respect to $(b,c)$ if

\begin{itemize}
\item[(C0)] $Q$ is non-negative and closed,
\item[(C1)]  $C_c (V) \subseteq D$,
\item[(C2)]  For any $u\in D$ and any $v\in C_c (V)$ the sum $\sum_{x\in V} u(x) \wl v(x) m(x)$ converges absolutely and the equality
$$Q(u,v) = \sum_{x\in V}  u (x) \overline{ \wl v (x) } m(x)$$
holds.
\end{itemize}
The selfadjoint operator $L$ induced by the form is then also said to satisfy $\Condition$.
\end{definition}

\textbf{Remark.} The requirement in (C0) that $Q$ is non-negative could be replaced by the assumption that $Q$ is bounded below (with appropriate changes). We assume that $Q\geq 0$  in order to simplify the notation later and not have to worry about some constants.

\bigskip

Theorem~\ref{regular} combined with \eqref{PI}  allows us to restate condition (C2) as follows.

\begin{coro} \label{regular-coro}  Let $(V,m)$ be a discrete measure space and $(b,c)$ a graph over $(V,m)$.  Let the form $Q$ with domain $D$ satisfy $\Condition$ with respect to $(b,c)$. Then, $D\subseteq \widetilde{F}$ and
$$Q(u,v) = \sum_{x\in V} \wl  u (x) \overline{v(x)} m(x)$$
holds for all $u\in D$ and $v\in C_c (V)$.
\end{coro}

The next proposition gathers some basic properties of  forms and operators satisfying $\Condition$ (and gives, in particular, that they are associated to a graph). Recall from Section \ref{Framework} the definition of $\langle \cdot,\cdot\rangle_Q $ via
$$\langle u, v\rangle_Q = Q (u,v) + \langle u, v\rangle.$$

\begin{prop} \label{Decomposition}  Let $(V,m)$ be a discrete measure space, $(b,c)$ a graph over $(V,m)$ and  $Q$ a form with domain $D$ satisfying $\Condition$ with respect to $(b,c)$. Then, the following properties hold:

(a)  $D(\qd)$ is a closed subspace of the Hilbert space $(D,\langle\cdot,\cdot\rangle_Q)$ and  $\qd$ and $Q$ agree on $D(\qd)$.

 (b) The selfadjoint operator $L$ associated to $Q$ is a restriction of $\wl$ and is non-negative.
\end{prop}

\begin{proof}  (a)  Note that $Q$ agrees with $\qd$ on $C_c (V)$ by the assumptions on $Q$. Thus, the closure of $C_c (V)$ with respect to  $\langle \cdot,\cdot\rangle_Q$ is exactly $D(\qd)$ and $D(\qd)$ is a closed subspace of the Hilbert space $(D,\langle\cdot,\cdot\rangle_Q)$.

\smallskip

(b) This is immediate from Corollary \ref{regular-coro}.
\end{proof}

A direct consequence of the previous proposition is the following maximality property of $\qd$.

\begin{coro}\label{Maximal}  Let $(V,m)$ be a discrete measure space and $(b,c)$ a graph over $(V,m)$. Then,  $Q\leq \qd$ holds for any form $Q$ satisfying $\Condition$ with respect to $(b,c)$.
\end{coro}

\bigskip

In order to state our main abstract result on the description of forms satisfying $\Condition$ we need one further piece of notation.

\begin{definition} (Harmonic function)  Let $(V,m)$ be a discrete measure space and $(b,c)$ a graph over $(V,m)$.  For a form $Q$ with domain $D$ satisfying $\Condition$ with respect to $(b,c)$,  the space of $1$-Q-harmonic functions $\mathcal{H}^{(Q)}$ is defined by
$$\mathcal{H}^{(Q)} :=\{u\in D : (\wl + 1) u = 0\}.$$
\end{definition}

\textbf{Remark.} By Corollary~\ref{eigenfunction}, the space $\mathcal{H}^{(Q)}$ could also be defined via `weak solutions', i.e.,
$$ \mathcal{H}^{(Q)} =\{u\in D: \sum_{x\in V} u(x) \overline{(\wl +1)  v(x)}  m(x) = 0\;\:\mbox{ for all $v\in C_c (V)$}\}.$$

\bigskip

Here is the main result of this section.

\begin{theorem}\label{general}  Let $(V,m)$ be a discrete measure space and $(b,c)$ a graph over $(V,m)$.  Let the form  $Q$ with domain $D$ satisfy $\Condition$ with respect to $(b,c)$.  Then, for $u\in D$,  the following assertions are equivalent:
\begin{itemize}

\item[(i)] $(\wl + 1) u = 0$.

\item[(ii)] $u$ is orthogonal to  $D(\qd)$ with respect to the inner product $\langle \cdot,\cdot\rangle_Q$.
\end{itemize}
Therefore, the Hilbert space $(D,\langle\cdot,\cdot\rangle_Q)$ can be decomposed as an orthogonal sum
$$D  =D(\qd) \oplus \mathcal{H}^{(Q)}.$$

\end{theorem}
\begin{proof} It suffices to show the equivalence of (i) and (ii). The remaining statement is then immediate.  Now, obviously, $(\wl + 1) u=0$ is equivalent to
$$ \sum_{x\in V}  (\wl + 1) u (x) \overline{ v(x)} m(x) = 0$$
for any $v\in C_c (V)$. By Corollary~\ref{regular-coro}, this is equivalent to
$$ 0 = Q(u,v) + \langle u,v\rangle = \langle u, v\rangle_Q $$
for all $v\in C_c (V)$. As $D(\qd)$ is the closure of $C_c (V)$ in $D$ with respect to $\langle \cdot, \cdot \rangle_Q$, we obtain the desired equivalence.
\end{proof}

\begin{coro} \label{general-coro} Let $(V,m)$ be a discrete measure space and $(b,c)$ a graph over $(V,m)$. Let the form  $Q$ with domain $D$ satisfy   $\Condition$ with respect to $(b,c)$. Define
$\mathcal{B}:= \mathcal{B} (Q):= D / D (\qd)$. Then, for each $v \in \mathcal{B}$ there exists a unique  $w = w_v \in D$ with
\begin{itemize}
\item $(\wl + 1) w =0$
\item $[w] = v$
\end{itemize}
and the map
$ \mathcal{B} \longrightarrow\mathcal{H}^{(Q)},$ $ v \mapsto w_v$,
is a bijection and even a unitary (if both vector spaces are equipped with the induced Hilbert space structure).
\end{coro}

\textbf{Remark.} One can think of $\mathcal{B} (Q)$ as a general type of boundary value of the elements of $D$.  Accordingly, the corollary gives the existence and uniqueness of a solution to a boundary value problem.

\bigskip

After this discussion of general features of the class of operators and forms satisfying $\Condition$, we now discuss important examples of such operators.  First, we show that the forms $\qn$ and $\qd$, and all closed forms between them, belong to this class.

\begin{prop}\label{PI-prop}  Let $(V,m)$ be a discrete measure space and $(b,c)$ a graph over $(V,m)$.  Then, any closed form $Q$  with $\qd \subseteq Q \subseteq \qn$  satisfies $\Condition$ with respect to $(b,c)$.  In particular, the selfadjoint operator $L$ associated to such a form $Q$ is a restriction of $\wl$.
\end{prop}
\begin{proof}  It suffices to show that $\qn$ satisfies $\Condition$.  By \eqref{PI}, it suffices to show that $D (\qn)\subseteq \widetilde{F}$.
To see this, we let $w \in D(\qn)$ and $B_x := \sum_y b(x,y)<\infty $ for each $x\in V$ and calculate
\begin{eqnarray*}
\sum_{y\in V} b(x,y) |w(y)| &\leq & \sum_{y\in V} b(x,y) |w(x) - w(y)| + \sum_{y\in V} b(x,y) |w(x)|\\
&\leq & \left(\sum_{y\in V} b(x,y)\right)^{1/2} \left(\sum_{y\in V} b(x,y) |w(x) - w(y)|^2\right)^{1/2} + B_x |w(x)|\\
&\leq & B_x^{1/2} {\qn}^{1/2} (w,w) + B_x |w(x)|.
\end{eqnarray*}
This gives the desired finiteness.
\end{proof}

\bigskip

We now turn to a situation in which we can explicitly describe all  Laplacians satisfying $\Condition$.

\smallskip

Recall from \cite{KL1} that  for  graphs $(b,c)$ over $(V,m)$ the following two conditions are equivalent:
\begin{itemize}
\item $\wl C_c (V) \subseteq \ell^2 (V,m)$.
\item $b(x,\cdot) / m(\cdot) \in \ell^2 (V,m)$ for all $x\in V$.
\end{itemize}
A graph satisfying one (and then both) of  these conditions will be said to satisfy the  \textit{finiteness condition (FC)}.

For such graphs, we  can define the \textit{minimal operator} $L_c$ to be the restriction of $\wl$ to $C_c (V)$ and the \textit{maximal  operator}  $L_M$ to be restriction of $\wl$ to
$$D (L_M):= \{u \in \ell^2(V,m):  \wl u\in \ell^2 (V,m)\}.$$
In this situation, the following consequence of \eqref{PI} holds (see \cite{KL1} for details):

\begin{prop}\label{maximaloperator} Let $(V,m)$ be a discrete measure space and $(b,c)$ a graph over $(V,m)$ satisfying $\mathrm{(FC)}$. Then, $L_M$ is the adjoint of $L_c$ and, in particular, the set of selfadjoint restrictions of $\wl$ is exactly the set of selfadjoint extensions of $L_c$.
\end{prop}

A special instance of graphs satisfying $\mathrm{(FC)}$ are locally finite graphs. Here,
a graph $(b,c)$ over $(V,m)$ is called \textit{locally finite} if, for any  $x\in V$, the set
$$ \{y\in V : b(x,y) >0\}$$
is finite. In this case, the previous proposition can be strengthened and it follows that $\widetilde{F} $ is equal to $C(V)$, $\wl$ maps $C_c (V)$ into itself and,  by \eqref{PI}, $\wl$  can easily be seen to be the adjoint of the restriction $L_c$  with respect to the dual pairing $C(V)\times C_c (V) \longrightarrow \CC$, $(u,v)\mapsto \sum_x u(x) v(x) m(x).$

\bigskip

Our characterization of all Laplacians satisfying $\Condition$ on graphs for which $\mathrm{(FC)}$ holds now follows:

\begin{theorem}\label{locally-finite} Let $(V,m)$ be a discrete measure space and $(b,c)$ a graph over $(V,m)$ satisfying $\mathrm{(FC)}$.  Let $L$ be a non-negative selfadjoint operator on $\ell^2 (V,m)$.  Then, the following assertions are equivalent:
\begin{itemize}
\item[(i)] $L$ and its associated form $Q$ satisfy $\Condition$.
\item[(ii)] $L$ is a restriction of $\wl$.
\end{itemize}
\end{theorem}
\begin{proof} The implication (i)$\Longrightarrow$(ii)  follows from Proposition~\ref{Decomposition} (and does not require (FC)). It remains to show the implication (ii)$\Longrightarrow$(i). This is a simple consequence of Proposition~\ref{maximaloperator} and \eqref{PI}.
\end{proof}

\smallskip


\section{The forms $\qd$ and $\qn$} \label{Forms}
 In this section we study how $\qd$ and $\qn$ differ from each other. The difference will turn out to be essentially given by solutions of
 $$ (\wl +1) u =0$$
 belonging to $D(\qn)$.  This will allow us to abstractly characterize when $\qd$ and $\qn$ agree.   We then turn to a more geometric description of this difference suggested by recent results of  \cite{TH,dVTHT}.

\bigskip

\begin{lemma} \label{lemma-one} \label{firstlemma} Let $(V,m)$ be a discrete measure space and $(b,c)$ a graph over $(V,m)$.  If $\qn\neq \qd$, then there exists a non-trivial, non-negative solution to $(\wl + 1) u =0$ in $D(\qn) \cap \ell^\infty (V)$.
\end{lemma}
\begin{proof} By $\qn \neq \qd$ we must have $\nl\neq \dl$ and then $(\nl +1)^{-1} \neq (\dl +1)^{-1}$.  As the vectors $\delta_x,$ $x\in V$, are total in $\ell^2 (V,m)$, there then exists an $x\in V$ such that
$$u:= ( (\nl +1)^{-1} - (\dl +1)^{-1}) \delta_x \neq 0,$$
where $\delta_x$ is the function in $\ell^2 (V,m)$ which vanishes everywhere except at $x$ where it is $1$.  As $\qn$ and $\qd$ are Dirichlet forms, both resolvents are contractions on $\ell^\infty (V)$ and the boundedness of $u$ follows.  Thus, $u$ belongs to $\ell^2 (V,m)\cap \ell^\infty (V)$. Moreover, $u$ belongs to $ D(\qn)$ as  $(\nl +1)^{-1}$ maps into $D(\nl)\subseteq D(\qn)$ and $(\dl +1)^{-1} $ maps into $D(\dl)\subseteq D(\qd) \subseteq D(\qn)$.  As both $\nl$ and $\dl$ are restrictions of $\wl$, we obtain that $u$ solves
$$ (\wl + 1) u =0.$$
Non-negativity of $u$ follows as $(\dl +1)^{-1} \delta_x $ is the smallest non-negative solution of $(\wl + 1) v = \delta_x$ by Theorem 11 of \cite{KL1}.
\end{proof}

By Proposition~\ref{PI-prop}, the form $\qn$ satisfies $\Condition$.  Thus, we can now specialize Corollary~\ref{general-coro} to obtain the following theorem on solving $(\wl +1) u =0$ in $D(\qn)$:

\begin{theorem}\label{main-qd-qn}  Let $(V,m)$ be a discrete measure space, $(b,c)$ a graph over $(V,m)$ and $\mathcal{B}^{(N)}:=D(\qn)/ D (\qd)$. Then, for each $v \in \mathcal{B}^{(N)}$  there exists a unique $w = w_v \in D(\qn)$ with
\begin{itemize}
\item $(\wl + 1) w =0$,
\item $[w] = v$.
\end{itemize}
\end{theorem}

As a corollary we obtain the following characterization of $\qn \neq \qd$.

\begin{coro} \label{char-qd-qn} (Characterization of $\qn \neq \qd$) Let $(V,m)$ be a discrete measure space, $(b,c)$ a graph over $(V,m)$ and $\mathcal{B}^{(N)}= D(\qn)/ D (\qd)$.   Then, the following assertions  are equivalent:
\begin{itemize}
\item[(i)] $\mathcal{B}^{(N)} \neq \{0\}$, i.e., $\qn \neq  \qd$.
\item[(ii)]  There exists a non-trivial solution of $(\wl +1) u =0$ in $D(\qn)$.
\item[(iii)] There exists a non-trivial solution of $(\wl +1) u =0$ in $D(\qn) \cap \ell^\infty(V)$.
\end{itemize}
\end{coro}
\begin{proof}
The implication (iii)$\Longrightarrow $(ii) is clear. The  implication  (ii)$\Longrightarrow $(i) (as well as the reverse direction)  follows from the previous result.  Note that the nontriviality of the solution $u$ in (ii) gives $[u]\neq 0$ by the uniqueness part of Theorem \ref{main-qd-qn}.
The implication (i)$\Longrightarrow$(iii) follows from  Lemma~\ref{lemma-one}.
\end{proof}

\medskip

A different angle to nontrivial solvability of $(\wl + 1) u =0$ is provided by the geometric context developed in \cite{TH,dVTHT} which we now recall.  The setting of \cite{TH,dVTHT} is concerned with locally finite graphs only.  However, the part that we need here works  in our situation with essentially the same proofs. For the convenience of the reader, we shortly indicate the corresponding  proofs.   For further discussion we refer to the cited works.

Assume that $(b,c)$ over $(V,m)$ is connected.  The length of a path $\gamma = (x_1,\ldots, x_n)$ is defined by
$$L(\gamma) :=\sum_{j=1}^{n-1} \frac{1}{b (x_j, x_{j+1})^{1/2}}.$$
Then, $d : V\times V\longrightarrow [0,\infty)$ given by
$$d(x,y) = \inf\{ L(\gamma) : \mbox{$\gamma$ is a path connecting $x$ and $y$}\}$$
provides a metric on $V$. Let $\widehat{V}$ be the metric completion of $V$ with respect to $d$ and let $\widehat{d}$ be the extension of $d$ to $\widehat{V}$.  Note that the completion $\widehat{V}$ may or may not agree with $V$.

The relevance of this metric comes from the fact that
\begin{equation} \label{rel}
|u(x) - u(y)|\leq \qn (u)^{1/2} d (x,y)
\end{equation}
for any $u\in D (\qn)$ where $\qn(u):= \qn(u,u)$. Indeed, for any path $\gamma = (x_1,\ldots, x_n)$ connecting $x$ and $y$, one easily sees by the Cauchy-Schwarz inequality and the subadditivity of the square root that
$|u(x) - u(y)|\leq \qn (u)^{1/2} L(\gamma)$. This gives the desired result.

\smallskip

Equation \eqref{rel}  allows one to extend any $u\in D (\qn)$ to a Lipschitz function $\widehat{u}$ on $\widehat{V}$.  We define $u_\infty$ to be the restriction of $\widehat{u}$ to $V_\infty :=\widehat{V}\setminus V$ if $\widehat{V}\setminus V\neq \emptyset$ and we define $u_\infty$ by $0$, otherwise. From the construction and some simple arguments we obtain a continuity property of the map $u\mapsto u_\infty$.  As this is not included in \cite{dVTHT} we discuss it explicitly  as follows:

\begin{lemma}\label{continuity-infty} Let $C(V_\infty)$ be the space of continuous functions on the metric space $\widehat{V}$.  Then, the map
$$ D(\qn)\longrightarrow C (V_\infty ), u\mapsto u_\infty, $$
is continuous when the set on the right hand side is given the topology of locally uniform convergence.
\end{lemma}
\begin{proof} Let $(u_n)$ be a sequence in $D (\qn)$ converging to $u\in D(\qn)$ with respect to $\langle \cdot,\cdot\rangle_{Q^{(N)}}$. Assume, without loss of generality, that there exists an $o\in V$ with $u_n (o) = u(o)$ for all $n\in \N$. Then, from \eqref{rel} we obtain
$$ | (u - u_n) (p)| = | (u - u_n) (p) - (u- u_n ) (o)|  \leq  \qn (u - u_n)^{1/2} \widehat{d} (p,o). $$
This gives the desired statement.
\end{proof}

Now, here comes the connection between  non-trivial solutions to $(\wl + 1) u  =0$ and the  boundary values $u_\infty$ of the  functions $u$ in $\qn$. This is our version of Theorem~2.1 of \cite{dVTHT}.

\begin{theorem}   \label{thm21}  Let $(V,m)$ be a discrete measure space and $(b,c)$ a graph over $(V,m)$. Let $\mathcal{B}^{(N)} = D(\qn)/ D (\qd)$.   Then, the map
$$P : \mathcal{B}^{(N)}\longrightarrow  \{ u_\infty : u\in D(\qn)\}, \quad [u]\mapsto u_\infty,$$
is well-defined, linear, continuous and onto.  In particular, to each $f\in D(\qn)$ there exists $w\in D(\qn)$ with  $(\wl + 1) w =0 $ and  $w_\infty = f_\infty$.  Furthermore, if $\widehat{V}$ is compact, then $P$ is injective.
\end{theorem}
\begin{proof} The second statement on the existence of $w$  follows from the first statement and Theorem~\ref{main-qd-qn}.  The last statement follows from a maximum principle as in \cite{dVTHT}. Thus, it suffices to show the first statement.  It is clear that $P$ is linear and onto (if it is well-defined).  Also, from the previous lemma, it is clear that it is continuous (if it is well-defined).  Thus, it remains to show that $P$ is well-defined.  Obviously, $u_\infty = 0$ for all $u \in C_c (V)$. Thus, by Lemma~ \ref{continuity-infty},  we obtain that $u_\infty = 0$ for all $u\in D(\qd)$.  Hence, $P$ is well-defined.
 \end{proof}

From this proposition, Corollary~\ref{char-qd-qn}, and Proposition \ref{esa-prop} we immediately infer the following corollary.

\begin{coro} Let $(V,m)$ be a discrete measure space and $(b,c)$ a graph over $(V,m)$. If there exists $f\in D(\qn)$ with $f_\infty \neq 0$, then $\qn \neq\qd$.  If, furthermore, \textup{(FC)} holds, i.e., $\wl C_c (V) \subset \ell^2 (V,m)$, then the restriction of $\wl$ to $C_c(V)$ is not essentially selfadjoint.
\end{coro}

\textbf{Remark.} (a) The statement on failure of essential selfadjointness in the  corollary is a generalization of Theorem~3.1 of \cite{dVTHT}. There, the statement is shown for locally finite graphs and $c\equiv 0$.   Our proof  provides a further piece of information  in that it shows that the  existence of $f\in D(\qn)$ with $f_\infty \neq 0$ implies that  $\qn \neq \qd$.

(b) One may wonder whether $\qn \neq \qd$ is, in fact, equivalent to existence of $f$ with $f_\infty\neq 0$. This, however, is not the case as can be seen by the example in Appendix~\ref{Counter}. In that situation, we have completeness of the graph (as this completeness does not depend  on $m$) and $\qn \neq \qd$.

\section{Characterizing Neumann  and Dirichlet Laplacians}\label{Neumann}
In our setting, it follows from Corollary~\ref{Maximal}, that the Dirichlet Laplacian is the largest operator satisfying $\Condition$.  This naturally raises the question whether a  corresponding characterization can be given  for the Neumann Laplacian. In this section, we show that this holds true  in the case of locally finite graphs.  More precisely, we study the set of all Markovian restrictions of $\wl$ and show that the Dirichlet Laplacian is the largest one and the Neumann Laplacian is the smallest one (Theorem \ref{char-Neumann}). These results (and their proofs) can be seen as analogues to results for the `usual' Laplacians (and diffusion-type operators) on sufficiently smooth subsets of Euclidean space as discussed in Section 3.3 of  \cite{Fuk}. As a corollary, we obtain a characterization of the agreement of $\qd$ and $\qn$ in terms of uniqueness of symmetric Markov processes associated to $\wl$ (Corollary \ref{char-markov-semigroup}).

\bigskip

We start with a definition.

\begin{definition} Let $(V,m)$ be a discrete measure space and $(b,c)$ a locally finite graph over $(V,m)$. Then, a non-negative selfadjoint restriction of $\wl$ is called Markovian if its associated form is a Dirichlet form. The set of all Markovian restrictions of $\wl$ is denoted by $\mathcal{E} = \mathcal{E} (b,c,m)$.
\end{definition}

\textbf{Remark.} If $(b,c)$ is locally finite, then, by Proposition~\ref{maximaloperator}, it follows that $L_M$ is the adjoint operator of $L_c$.  Therefore, any selfadjoint $L$ is a restriction of $\wl$ if and only if it is an extension of $L_c$ and, in this case,
$$L_c \subseteq L \subseteq  L_M$$
holds.  We can  therefore think of restrictions of $\wl$ as extensions of $L_c$  and this explains our notation  $\mathcal{E}$ for a set of restrictions.

\medskip

\begin{theorem} \label{char-Neumann}  Let $(V,m)$ be a discrete measure space and $(b,c)$ a locally finite graph over $(V,m)$. Then,
$$ \qn \leq Q \leq \qd$$
holds for any form $Q$ associated to a Markovian restriction $L$ of $\wl$.
\end{theorem}

The proof of this theorem is given  after a series of intermediate claims.  We will assume that we are given a locally finite graph $(b,c)$ over $(V,m)$ throughout.  
Moreover, by a slight abuse of notation, we will write
 $$\langle u, v\rangle := \sum_{x\in V} u(x) \overline{v(x)} m(x)$$
for all $u\in C(V)$ and $v\in C_c (V)$.

\smallskip

By definition, the form associated to $L\in \mathcal{E}$ is a  Dirichlet form. This has the following consequences which will be repeatedly used in the sequel (see \cite{Fuk, Dav2} for proofs): For any $\beta >0$, the resolvent
 $$G_\beta := ( L + \beta)^{-1}$$ is positivity preserving, i.e., maps non-negative functions to non-negative functions. Moreover, for all $1\leq p \leq\infty$, the map  $G_\beta$ extends  uniquely  to a map on $\ell^p (V,m)$, again denoted by $G_\beta$,  with norm not exceeding  $\frac{1}{\beta}$ and satisfying
$$G_\beta u = \lim_{n\to \infty}  G_\beta u_n $$
whenever $u_n, u \geq 0$ with $u_n \to u$ monotonously increasing. The $G_\beta$  are obviously selfadjoint  on $\ell^2 (V,m)$ and their extensions have the  following symmetry property
$$\langle G_\beta u, v\rangle = \langle u, G_\beta v\rangle$$
for any $ u\in \ell^p (V,m)$, $p\geq 1$, and $v\in C_c (V)$.

\smallskip

\begin{prop} \label{convergence-at-zero}  Let $L\in \mathcal{E}$ and $G_\beta = (L + \beta)^{-1}$ for $\beta>0$. For any $u\in \ell^p (V,m)$ and $v \in C_c (V)$ the equation
$$\lim_{\beta \to \infty} \beta \langle u - \beta G_\beta u, v\rangle = \langle \wl u, v\rangle $$
holds.
\end{prop}
\begin{proof} It suffices to show
$$\lim_{\beta \to \infty} \beta \langle u - \beta G_\beta u, v\rangle = \langle  u, \wl v\rangle. $$
Then, the claim follows from \eqref{PI} as $\widetilde{F} = C(V)$ due to local finiteness. We calculate
\begin{eqnarray*}
\beta \langle u  -  \beta G_\beta u, v\rangle &=& \beta \left( \langle u, v\rangle - \beta \langle G_\beta u, v\rangle \right)\\
(!) &=& \beta \langle G_\beta u, L v\rangle\\
&=&\langle u, \beta  G_\beta  L v\rangle \\
(!!) &\to & \langle u, L v\rangle.
\end{eqnarray*}
As $L\subseteq L_M$ (see above), this shows the desired statement.

Here, (!!)  follows from the  spectral theorem. The statement
 $(!)$ follows as, obviously,
$$ \langle G_\beta u, (L + \beta) v\rangle =  \langle u, v\rangle$$
and hence
$$ \langle G_\beta u, L v\rangle = \langle u, v\rangle - \beta \langle G_\beta u, v\rangle.$$
This finishes the proof.
\end{proof}

\begin{prop} \label{qb} Let $L\in \mathcal{E}$ and $G_\beta = (L + \beta)^{-1}$ for $\beta>0$.  Then, for any real-valued $u\in D(Q)$
$$ \beta \langle u - \beta G_\beta u, u\rangle  = \frac{1}{2} \langle f_\beta, 1\rangle$$
where the non-negative function $f_\beta : V\longrightarrow \R$  belonging to $\ell^1 (V,m)$ is given by
\begin{align*}
f_\beta (x) &= \beta^2 G_\beta (u(x) 1 - u)^2 (x) + 2 \beta u(x)^2 ( 1 - \beta G_\beta 1 (x))\\
 &= - \beta \left(u^2 (x) - \beta (G_\beta u^2) (x) \right) + 2 \beta u(x) \left( u(x) - \beta G_\beta u (x) \right) \\
 & \qquad + \beta u^2 (x) ( 1 - \beta G_\beta 1 (x)).
 \end{align*}
Here, $u(x) 1$ denotes the constant function with value $u(x)$ on $V$.
\end{prop}

\begin{proof} We start by discussing the definition of $f_\beta$: We first note that both expressions for $f_\beta$ make sense as $G_\beta$ is applied to (sums of) elements from $\ell^p (V,m)$ for $1 \leq p \leq \infty$.  Note that the first representation gives that  $f_\beta$ is non-negative  and the second representation gives that $f_\beta$ belongs to $\ell^1 (V,m)$.

Finally,  the claimed equalities follow by direct computations. These use
$$ m(x) G_\beta w (x) =  \langle G_\beta w, \delta_x\rangle = \langle w, G_\beta \delta_x\rangle$$
for any $w$ which is a sum of functions in $\ell^p (V,m)$ and for $\delta_x $, the characteristic function of $\{x\}$.
\end{proof}

\begin{prop} \label{decomposition}  For any $L\in \mathcal{E}$ with associated form $Q$ the Hilbert space $D(Q) $ with inner product $Q (\cdot,\cdot) + \langle \cdot,\cdot\rangle$ is the orthogonal sum
$$D(Q) = D (\qd) \oplus \mathcal{H}^{(Q)},$$
where $\mathcal{H}^{(Q)}  = \{u\in D(Q) : (\wl + 1) u = 0\}$.  In particular, any $u\in D(Q)$ can be decomposed uniquely as $u = u_0 + h$ with $u_0 \in D(\qd)$ and $h \in \mathcal{H}^{(Q)}$ and
$$ Q(u,u) + \langle u, u\rangle = \qd (u_0, u_0) + Q (h,h) + \langle u_0, u_0 \rangle + \langle h, h\rangle. $$
\end{prop}
\begin{proof} The first statement is an immediate consequence of Theorem~\ref{general} (see (ii) of Lemma 3.3.2 of \cite{Fuk} as well).
The second statement is an immediate consequence of the first statement.
\end{proof}

After these preparations, we are now ready to give a proof of the main result of this section.

\begin{proof}[Proof of Theorem~\ref{char-Neumann}:] To avoid tedious but non-essential terms we assume that $c\equiv 0$.  The statement on the Dirichlet operator $\dl$ is clear and has already been discussed in the introduction to this section. We show the statement on the Neumann operator. Thus, let $L\in \mathcal{E}$ be given and $Q$ be the associated Dirichlet form.  By Proposition~\ref{decomposition} (applied to both $Q$ and $\qn$) it suffices to show that
\begin{equation}\label{aim}
Q (u,u) \geq \qn (u,u)
\end{equation}
for all real-valued $u\in D(Q)$ with $(\wl + 1) u = 0$. We will investigate the left hand side and the right hand side of \eqref{aim} separately. To do so we define  $T : V\longrightarrow \R$ by
$$ T(x):= -  \sum_{y\in V}  b(x,y) (u(x) - u(y)) u(y).$$
Note that $T(x)$ is well-defined as $(b,c)$ is locally finite.

\smallskip

\textit{Right hand side of \eqref{aim}.} As we do not even know that $u\in D(\qn)$ we have to exercise some care. However, by Fubini's theorem and the local finiteness of $b$, the expression
\begin{eqnarray*} \qn (u,u) & = &  \frac{1}{2}\sum_{x,y\in V} b(x,y) (u(x) -  u(y)) (u(x) - u(y)) \\
&=&  \frac{1}{2}\sum_{x\in V} \left( \sum_{y\in V} b(x,y) (u(x) - u(y)) u(x) - \sum_{y\in V} b(x,y) (u(x) -u(y)) u(y)\right)
\end{eqnarray*}
is well-defined  (i.e., either converges absolutely or diverges to $\infty$) and all inner sums converge.   Now,  using $(\wl + 1) u =0$  and the absolute convergence of $\sum_x u(x)^2 m(x)$, we obtain
\begin{equation}\label{calculating-qn}
 0 \leq \qn (u,u) = \frac{1}{2} \sum_{x\in V} (- u(x)^2 m(x)  + T(x)),
 \end{equation}
 where the sum is well-defined, i.e., either converges absolutely  or diverges to $\infty$.

\smallskip

\textit{Left hand side of \eqref{aim}.} By the spectral theorem and Proposition~\ref{qb} we have that
$$ Q (u,u) \geq \beta\langle u - \beta G_\beta u, u\rangle  \geq \frac{1}{2} \langle f_\beta, v\rangle \geq 0
$$
for any $v\in C_c (V)$ with $0 \leq v\leq 1$ and any $\beta >0$. Here, we use that  $f_\beta \geq 0$.
Hence, taking $\beta \to \infty$, by Fatou's lemma, Proposition~\ref{convergence-at-zero} and the second expression for $f_\beta$ in Proposition~\ref{qb}, we obtain that
$$ Q(u,u) \geq \sum_{x\in V} v(x)\left (- \frac{1}{2} (\wl u^2) (x) + u \wl u (x) \right) m(x)\geq 0.$$
Now, a direct computation using $(\wl + 1) u =0$  shows that
$$(\wl u^2) (x) = -  u^2 (x)  -  \frac{1}{m(x)} T(x)$$
and
$$ u (x) \wl u (x) = -  u^2 (x).$$
Putting this together, we find that
$$ Q (u,u)  \geq  \frac{1}{2} \sum_{x\in V}  v(x) ( -  u^2 (x) m(x)  + T (x))\geq 0.$$
As $v\in C_c (V)$ with $0\leq v\leq 1$ was arbitrary, we obtain, in particular, that
$$  (- u^2 (x) m(x)  + T (x))  \geq 0$$
for all $x\in V$.  This shows that we can take a limit over $v$ with $ 0 \leq v \leq  1$ and $v\to 1$ pointwise to obtain
$$ Q (u,u) \geq \frac{1}{2}\sum_{x\in V} (- u^2 (x) m(x)  + T (x)).$$
Comparing this with \eqref{calculating-qn} we obtain that
$$ Q (u,u) \geq \qn (u,u)$$
 and the desired statement follows. This finishes the proof.
\end{proof}

 \bigskip

 \textbf{Remarks.} (a) Note that great care has to be exercised when plugging $u$ with $(\wl + 1) u =0$ into $\qn$ as, formally,
  $$0 \leq \qn (u,u) = \langle \wl u, u\rangle = - \langle u, u\rangle \leq 0$$
 giving $\qn (u,u)=0$ for all such $u$ (which would imply that $u=0$ whenever the graph is connected and  $m(V) = \infty$).

(b) In general, there will exist Dirichlet forms between $\qd$ and $\qn$.  One way to generate such a form is given as follows: Choose an arbitrary subset $A\subseteq  V$. Then, define
$$D_A^{'} :=\left\{ u\in D(\qn) : \sharp\{ x \in A : u(x) \neq 0\} <\infty \right\}$$
and let $D_A$ be the closure of $D_A^{'}$ in the Hilbert space $(D (\qn), \langle \cdot,\cdot\rangle_{\qn})$. Then, the restriction of $\qn$ to $D_A$ will be a Dirichlet form by Theorem~3.1.1 of \cite{Fuk}. In general, it will differ from both $\qn$ and $\qd$.

(c) 
The previous remark shows that there exist Dirichlet forms between $\qd$ and $\qn$. Based on the considerations of this paper one might try and characterize all of these forms (via boundary conditions). We consider this an interesting open problem.

(d) The considerations above use the local finiteness of the graph in various places.  It should be interesting to find out whether a similar result holds in the general case as well.

\bigskip

Our considerations give another characterization of $\qd = \qn$ in the case of locally finite graphs. To state the characterization we introduce one more piece of  (standard)  notation: A map $P$ from $[0,\infty)$ into the set of selfadjoint bounded operators on $\ell^2 (V,m)$ is called a \textit{strongly continuous symmetric semigroup} if it has the form  $P_t = e^{-t L}$ for a selfadjoint  $L$  which is bounded below.   Here, $e^{-t L}$ is understood in the sense of spectral calculus for selfadjoint operators.  The operator $L$ is called the \textit{generator} of the semigroup.  If the form associated to $L$ is a Dirichlet form, the semigroup is called \textit{Markovian}.

\begin{coro}\label{char-markov-semigroup}
Let $(V,m)$ be a discrete measure space and $(b,c)$ a locally finite graph over $(V,m)$.  Then, the following assertions are equivalent:
\begin{itemize}
\item[(i)] $\qd = \qn$.
\item[(ii)] There exists  only one   strongly continuous symmetric Markovian semigroup with generator contained in $\wl$.
\end{itemize}
\end{coro}

\textbf{Remark.} It should be interesting to find out to what extent a similar result may hold for more general Dirichlet forms.

\section{The equation $(\wl + 1) u =0$}\label{Equation}
In Section~\ref{Forms} we have seen that the set of solutions of $(\wl +1) u =0$ in $D(\qn)$ describes the difference between $\qn$ and $\qd$.  In particular, the disagreement of $\qn$ and $\qd$ was characterized in terms of nontrivial solvability of $(\wl + 1) u =0$ in $D (\qn)$.
In this section we put these results in perspective by discussing the nontrivial solvability of $(\wl + 1) u =0$ in  the spaces  $\ell^2 (V,m)$ and $\ell^\infty (V)$. This  turns out to be related to essential selfadjointness and stochastic completeness, respectively.
As a consequence, we obtain some immediate connections between the agreement of $\qn$ and $\qd$, essential selfadjointness and stochastic completeness. By examples, we show that no further implications hold in general.

\bigskip

 Before we start the discussion let us note that the number one in the equation $(\wl +1) u = 0$  does not play any special role. It could be replaced by any positive number $\alpha$. Then, virtually the same arguments apply to solutions of $(\wl + \alpha) u =0$. In fact, the arguments apply to  any number $\alpha$ with $-\alpha$ smaller than the infimum of the spectrum of $\nl$. We stick to the case $\alpha =1$ for convenience only.

 \bigskip

We now turn to the concept of stochastic completeness.  Recall that $(V,b,c,m)$ with $c\equiv 0$ is called \textit{stochastically complete} if $$M_t := e^{-t \dl} 1 \equiv 1$$
  for all $t\geq0$. This can be shown to be equivalent to  $(\wl +1) u =0$ not having a non-trivial solution in $\ell^\infty(V)$ (and to various further statements) \cite{Woj2,KL1,Hua}. It turns out that this type of characterization  can be extended to the case $c\not \equiv 0$ if one is willing to modify $M$.  More precisely, in the general case (with not necessarily vanishing $c$),  one defines,  for each $t\geq 0$,  the function $M_t$ on $V$ by
$$M_t :=e^{-t \dl } 1 + \int_0^t e^{-s \dl} \frac{c}{m} ds.$$
Here, for the non-negative $c/m$, the function  $e^{-t \dl} c/m$ is defined as a limit by approximating $c/m$ from below by non-negative functions in $C_c (V)$ (see \cite{KL1}). The function  $M_t$ turns out to be finite with values between $0$ and $1$.  Note that the function  agrees with our earlier definition of $M_t$ if $c\equiv 0$. We then say that $(V,b,c,m)$ satisfies $(SC_\infty)$ if $M_t \equiv 1$ and speak of $(SC_\infty)$ as \textit{stochastic completeness at infinity}. As shown in \cite{KL1,KL2} the following holds.

\begin{theorem}(Characterization of stochastic completeness)
Let $(V,m)$ be a discrete measure space and $(b,c)$ a graph over $(V,m)$. Then, the following assertions are equivalent:
\begin{itemize}
\item[(i)] $(V,b,c,m)$ is stochastically complete at infinity, i.e., $M_t \equiv 1$ for all $t\geq0$.
\item[(ii)]  There does not exist a non-trivial solution of $(\wl +1) u =0$ in $\ell^\infty (V)$.
\end{itemize}
\end{theorem}

\textbf{Remark.}  Note that $M$ is defined using the Dirichlet operator $\dl$.  Analogously, one might define
$$M_t^{(N)}:=e^{-t \nl } 1 + \int_0^t e^{-s \nl} \frac{c}{m} ds.$$
We then infer that $M\equiv 1 \Longrightarrow M^{(N)} \equiv 1$ (as $M\equiv 1$ implies $\qn = \qd$ by  Lemma~\ref{firstlemma} and the previous theorem).  However, the reverse implication that $M^{(N)} \equiv 1\Longrightarrow M\equiv 1$ does \textbf{not} hold. To see this, we can consider the   example of Appendix~\ref{Counter} (see  \cite{KL1} as well)  with $c \equiv 0$, $m(V) <\infty$ and $\qn \neq \qd$. By $\qn \neq \qd$, Lemma~\ref{firstlemma} and the preceding theorem, we infer the failure of $M\equiv 1$. On the other hand, by $m(V) <\infty$ we obtain that $1$ is eigenfunction of $\nl$. Thus, $M^{(N)} = e^{- t \nl} 1 \equiv 1$.  This shows that, in terms of processes, stochastic completeness cannot be defined with the `Neumann-process'. This seems  worth noting  as the characterization of stochastic completeness via (un)boundedness of solutions of $(\wl +1) u =0$ does not refer to any specific selfadjoint realization of $\wl$.

\bigskip

We now turn to essential selfadjointness. The following result essentially deals with the deficiency index being zero. In the context of graph Laplacians it could be derived from the considerations of \cite{KL1}. We include a proof for completeness.

\begin{theorem} \label{Char-ES} (Characterization of essential selfadjointness) Let $(V,m)$ be a discrete measure space and $(b,c)$ a graph over $(V,m)$.  Assume $\mathrm{(FC)}$, i.e.,  $\wl C_c (V) \subseteq \ell^2 (V,m)$.
Then, the following assertions are equivalent:
\begin{itemize}
\item[(i)] The restriction of $\wl$ to $C_c (V)$ is essentially selfadjoint.
\item[(ii)]  There does not exist a non-trivial solution of $(\wl +1) u =0$ in $\ell^2 (V,m)$.
\end{itemize}

\end{theorem}
\begin{proof} Recall, from Section~\ref{Laplacians}, the definition of the operator $L_c$ as the restriction of $\wl$ to $C_c (V)$ and the maximal operator $L_M$ as the   restriction of $\wl$ to the set of all $u\in \ell^2 (V,m)$ with $\wl u\in \ell^2(V,m)$. Then, by Proposition~\ref{maximaloperator},   $L_c$ is a symmetric non-negative operator with adjoint $L_M$. As $L_c$ is non-negative, essential selfadjointness is then equivalent to triviality of the kernel of $L_M + 1$ by standard theory.
\end{proof}

We note the following consequence of  Lemma~\ref{firstlemma} and the considerations above.

\begin{coro}\label{con}
Let $(V,m)$ be a discrete measure space and $(b,c)$ a graph over $(V,m)$.

(a) If the graph is stochastically complete at infinity, then $\qn = \qd$.

(b) If $\wl$ maps $C_c (V)$ to $\ell^2 (V,m)$ and is essentially selfadjoint, then $\qn = \qd$.
\end{coro}

\bigskip

Abbreviating stochastic completeness at infinity by $S. C. $ and essential selfadjointness by $E. S.$ we
can summarize the preceding considerations  as follows:

\begin{align*}
\left.
\begin{array}{l}
\mbox{$S.C.$}\\
\mbox{$E.S.$}\\
\mbox{$Q^{(N)} = Q^{(D)}$ }
\end{array}
\right\}
\mbox{ No nontrivial solution to $(\widetilde{L} + 1) u = 0$ in }
\left\{
\begin{array}{r}
\mbox{$\ell^\infty (V)$}\\
\mbox{$\ell^2 (V,m)$ and (FC)}\\
\mbox{$D(\qn) \cap \ell^\infty (V)$}\\
\end{array}\right.
\end{align*}

In particular: $E.S.$  $\Longrightarrow$ $ \qn = \qd $ $ \Longleftarrow$ $S.C.$

\bigskip

 This shows some connections between stochastic completeness, essential selfadjointness and $\qn = \qd$. It turns out that no further implications hold, i.e., stochastic completeness at infinity and essential selfadjointness are independent. In particular, neither in (a) nor in (b) of Corollary~\ref{con}  does the reverse implication hold. This is now  shown by a series of examples:

\medskip

\textbf{Example 1 and Example 2} (Graphs satisfying $E. S.$ and  $S.C$. and graphs satisfying  $E. S.$ without $S.C.$, respectively) We consider graphs with  $m \equiv 1$, $c\equiv 0$ and $b$ taking values in $\{0,1\}$ only. Then, as shown in \cite{KL1, Woj1},  essential selfadjointness holds due to the assumption that $m\equiv 1$.

 More specifically, we will now  even further restrict attention to  radially symmetric rooted trees. Thus, we are given a tree with a root $o$ and all vertices with distance $n$ to the root have the same degree $d_n$. Then, as shown in \cite{Woj1}, the corresponding models will satisfy $S.C.$ if and only if
$$\sum_{n}\frac{1}{d_n} = \infty.$$
Thus, within the class of radially symmetric rooted trees, we can easily find examples satisfying $E. S.$ together with $S.C$. and examples satisfying  $E. S.$ without $S.C.$

\medskip

\textbf{Example 3} (Graphs satisfying  neither $S.C.$ nor $E. S$.)  The example in Appendix~\ref{Counter} (see  Section $4$ of \cite{KL1} as well) gives a situation  with  $m(V)<\infty$, $c \equiv 0$, and $1\not\in D(\qd)$ implying, in particular, that $\qn \neq \qd$. Thus, in this example, we have neither essential selfadjointness nor stochastic completeness.

\medskip

\textbf{Example 4} (Graphs satisfying $S.C.$ without $E.S.$)  Based on the  first example of Section $4$ of \cite{KL1} we may give such a graph as follows:
Let $V=\Z$ and  $b(x,y)=1$ if $|x-y|=1$ and zero otherwise. The weighted graph $(b,0)$ over $(V,1)$ gives rise to the bounded operator ${\Delta}:\ell^2(\Z)\to\ell^{2}(\Z)$ which is the restriction of the formal operator $$(\widetilde{\Delta}w)(x)=-w(x-1)+2w(x)-w(x+1)$$
to $\ell^2(\Z):= \ell^2(\Z,1)$.
As ${\Delta}$ is bounded (see, e.g., Theorem \ref{char-boundedness} in Section \ref{Boundedness} below), it is essentially selfadjoint. Moreover, it is well-known that $(\Z,b,0,1)$ is stochastically complete (see Examples~1 and~2 and compare, for instance, \cite[Theorem~2.10]{DM}) and, therefore, there is no solution to $(\widetilde{\Delta}+1)w=0$ in $\ell^{\infty}(\Z)$.

Let $\lm:=\arccosh (3/2)$.
As $e^{\lm}+e^{-\lm}-2=1$, one  checks that
$$u:\Z\to[0,\infty), \; x\mapsto e^{\lm x}$$
 is a solution to $(\widetilde{\Delta}+1)u=0$.  Choose $\ph\in\ell^{1}(\Z)$ with $\ph(x) > 0$ for all $x \in \Z$.  Let the  measure $m$ on $\Z$ be given by
\begin{align*}
 m(x) =\min \left \{1,\frac{\ph (x)}{u^2 (x)} \right \}
\end{align*}
and the  map $c:\Z\to[0,\infty)$ via
\begin{align*}
c(x):=\max \left \{ 0,\frac{u^2 (x)}{\ph(x)}-1 \right \}m(x).
\end{align*}
Then,
$$ 1 \equiv c + m$$
by construction.
The graph $(\Z,b,c,m)$ induces the formal operator
$$(\widetilde{L}w)(x) =\frac{1}{m(x)}({-w(x-1)+2w(x)-w(x+1)})+\frac{c(x)}{m(x)}w(x),$$
or, formally, $\widetilde L=\frac{1}{m}(\widetilde{\Delta}+c)$.
As $1 \equiv c + m $, one directly checks that  $w:\Z\to\R$ solves
$$\mbox{$(\widetilde{L}+1)w=0$\; if and only if \; $(\widetilde{\Delta}+1)w=0$.}$$

We conclude the following: The function $u:x\mapsto e^{\lm x}$ belongs to  $\ell^2(\Z,m)$ by the choice of $m$ and  $(\widetilde{L}+1)u=0$ since $(\widetilde{\Delta}+1)u=0$. Therefore,  the restriction of $\widetilde{L}$ to $C_c(\Z)$ is not essentially selfadjoint.  On the other hand, there is no solution to $(\widetilde{L}+1)w=0$ in $\ell^{\infty}(\Z)$ since there is no solution to  $(\widetilde{\Delta}+1)w=0$ in $\ell^{\infty}(\Z)$ (see above).  Hence, $(\Z,b,c,m)$ is stochastically complete.

\section{Maximum principle and characterization of positivity improvement}\label{Maximum}
In this section we present a maximum principle and use it to characterize positivity improvement of a positivity preserving  semigroup of the form $(e^{- t L})_{t\geq 0}$ with $L\subseteq \wl$. For the Dirichlet Laplacian this has already been done in \cite{KL1}.

\bigskip

\begin{theorem}(Maximum principle) Let $(V,m)$ be a discrete measure space, $(b,c)$ a connected graph over $(V,m)$ and $\wl = \wl_{(b,c,m)}$. Let $w$ be a real-valued solution of $\wl  w \leq  0$. If $w$ attains its maximum and this maximum is non-negative, then $w$ is constant and even $w \equiv 0$ if $c\not \equiv 0$.
\end{theorem}
\begin{proof} Let $x\in V$ be given such that $w$ attains its non-negative  maximum at $x$. From
$$0 \geq  \wl  w (x) = \frac{1}{m(x)}\sum_{y\in V} b(x,y) (w(x) - w(y)) + \frac{c(x)}{m(x)} w(x) $$
we then infer that all $y\in V$ with $b(x,y) >0$ must have $w(x) = w(y)$. Inductively, we obtain the constancy of  $w$. Now, a second look at $\wl w \leq 0$ shows the last part of the statement.
\end{proof}

\begin{coro} Let $(V,m)$ and $(b,c)$ be as in the previous theorem. Let $u$ be a non-negative solution to $(\wl + 1) u \geq 0$. Then, either $u$ is strictly positive  or $u\equiv 0$.
\end{coro}
\begin{proof} This is a direct consequence of the previous theorem applied with  $w=-u$ and $\wl$ replaced by $\wl + 1$.
\end{proof}

Recall that a bounded operator $A$ on $\ell^2 (V,m)$ is \textit{positivity preserving} if it maps non-negative functions to non-negative functions. It is called \textit{positivity improving} if it maps non-negative functions which do not vanish identically to  strictly positive functions. A semigroup $(e^{-t L})_{t\geq 0}$ is said to be positivity preserving and positivity improving, respectively, if, for every $t>0$, $e^{-t L}$ has the corresponding property.

\begin{theorem} \label{CPI} (Characterization of positivity improvement)
Let $(V,m)$ be a discrete measure space and $(b,c)$ a graph over $(V,m)$.   Let $L$ be a selfadjoint non-negative restriction of $\wl$ such that the associated semigroup $(e^{-t L})_{t\geq 0}$ is positivity preserving.
Then, the semigroup $(e^{-tL})_{t\geq 0}$ is positivity improving if and only if $(b,c)$ is connected.
\end{theorem}
\begin{proof}
It is clear that the semigroup cannot be positivity improving if the graph is not connected.

Let us now turn to the other implication. Thus, we assume that the graph $(b,c)$ is connected.  By general principles, it suffices to show that the resolvent $(L + 1)^{-1}$ is positivity improving. So, let $u\geq 0$ with $u\not \equiv 0$ be given and consider the function $v:= (L+1)^{-1} u$. Then, $v$ is non-negative as $e^{-tL}$, and thus $(L+1)^{-1}$, is positivity preserving and satisfies $(\wl +1) v =u\geq 0$ as $L$ is a restriction of $\wl$.  Now, the desired positivity follows from the   previous corollary.
\end{proof}

\section{An analogue to a theorem of Li}\label{Analogue}
Whenever $L$ is a non-negative operator on $\ell^2(V,m)$ we can form the associated semigroup $e^{-t L}$. By the discreteness of $V$ these operators have a kernel, i.e., there exists a map
$$ p : [0,\infty)\times V\times V\longrightarrow \CC$$
with
$$ e^{-t L} f(x) = \sum_{y\in V} p_t (x,y) f(y) m(y)$$
for all $f\in \ell^2 (V,m)$.  Thus, with the characteristic function $\delta_x$ of $x\in V$ we obtain
$$ \langle \delta_x, e^{-t L} \delta_y\rangle  = m(x) m(y) p_t (x,y)$$
for all $x,y\in V$.
If $L$ arises from a Dirichlet form, then $p$ must be non-negative with $\sum_{y} p_t (x,y) m(y) \leq 1$.  In this case, estimates  of this kernel are of particular interest.  Some basic estimates are discussed in the main result of this section.  The result is taken from \cite{KLVW}, following \cite{CK, Sim}. We present an alternative proof for part (b), which is is known as Theorem of Li in the context of manifolds (after \cite{Li}).

\bigskip

\begin{theorem}\label{Li}
Let $(V,m)$ be a discrete measure space and $(b,c)$ a connected graph over $(V,m)$.  Let $L$ be a selfadjoint non-negative restriction of $\wl$ such that the associated semigroup $(e^{-t L})_{t\geq 0}$ is positivity preserving.  Furthermore, let $E_0$ be the  infimum of the spectrum of $L$.

(a) There exists a unique non-negative $\Phi$ on $V$ such that
$$e^{t E_0} p_t (x,y) \to \Phi (x) \Phi (y), t\to \infty$$
for all $x,y\in V$. Here, $\Phi \equiv 0$ if $E_0$ is not an eigenvalue and $\Phi$ is the unique normalized positive eigenfunction to $E_0$,  otherwise.

(b)  The kernel $p$ of $e^{-t L}$ satisfies
 $$ \frac{\log p_t (x,y)}{t} \to -E_0, \qquad t\to \infty$$
for all $x,y\in V$.
\end{theorem}

\begin{proof}
\smallskip

Part (a) can  be obtained as a simple  consequence of the spectral theorem (see \cite{Sim,KLVW}) as follows:
Let $P$ be the projection onto the eigenspace of $E_0$, i.e., $P =0$ if $E_0$ is not an eigenvalue and $P = \langle \Phi, \cdot \rangle \Phi$ otherwise.

 Then, the spectral theorem gives
 \begin{eqnarray*}
m(x) m(y)   | e^{t E_0} p_t (x,y) - \Phi (x) \Phi (y) |& = &  |\langle {\delta}_x, ( e^{t E_0} e^{- t L} - P) {\delta}_y\rangle|  \\
  & = & \left| \int_{[E_0,\infty)} (e^{- t (s -E_0)} - 1_{\{E_0\}} (s) )d\rho (s) \right|\\
  & \to &  0.
  \end{eqnarray*}

Here, $\rho$ is the spectral measure of $L$ associated to ${\delta}_x$ and ${\delta}_y$ characterized by
$$\langle  {\delta}_x, f(L) {\delta}_y\rangle = \int f (s) d\rho (s)$$
for all continuous bounded real valued  $f$ on $\R$.

\medskip

(b)  Let $\tilde{\delta}_x$ with
$$\tilde{\delta}_x (y) = \frac{1}{\sqrt{m(x)} }\delta_x(y)$$
be given, where $\delta_x$ is the characteristic function of $\{x\}$.
  Obviously, $\{\tilde{\delta}_x\}_{x\in V}$ forms an orthonormal basis of $\ell^2 (V,m)$ consisting of non-negative functions which do not vanish identically.  As $(b,c)$ is connected, we infer, from Theorem~\ref{CPI}, that the semigroup $(e^{-t L})_{t\geq 0}$ is positivity improving. Thus, for all $t>0$ and $x,y\in V$, the numbers
$$a_t (x,y):=\langle \tilde{\delta}_x, e^{-t L} \tilde{\delta_y} \rangle,\:\; a_t (x):= a_t (x,x)$$
are positive and satisfy
$$ a_{t + s} (x) = \langle e^{- t L}  \tilde{\delta_x}, e^{- s L} \tilde{\delta_x}\rangle = \sum_{y\in V} \langle e^{ - t L} \tilde{\delta}_x, \tilde{\delta}_y\rangle \langle \tilde{\delta}_y, e^{- s L} \tilde{\delta_x}\rangle \geq a_t (x) a_s (x).$$
By a similar reasoning,
$$a_{t-1} (x) a_1 (x,y) \leq a_t (x,y) \leq \frac{1}{a_{1} (y,x)} {a_{t+1} (y)}$$
for all $x,y\in V$ and $t>1$.
The first inequality  gives the existence of
$$ \lim_{t\to \infty} \frac{\log a_t (x)}{t} $$
for each $x\in V$ by standard subadditive reasoning. The second line of inequalities then gives that this limit does not depend on $x$ and, in fact,
$$\lim_{t\to \infty} \frac{\log a_t (x,y)}{t} = E$$
holds with some real  $E$ for all $x,y\in V$.
As  $ \sqrt{ m(x) m(y)}   p_t (x,y) = a_t (x,y)$, we obtain the convergence of $\frac{\log p_t (x,y)}{t}$ to $E$ for $t\to \infty$ as well. As this holds for all $x,y\in V$, we obtain  $E= E_0$.
\end{proof}

\begin{coro} Let $(V,m)$ be a discrete measure space and $(b,c)$ a connected graph over $(V,m)$. Let $L$ be a selfadjoint non-negative restriction of $\wl$ such that the associated semigroup $(e^{-t L})_{t\geq 0}$ is positivity preserving. Assume that $m(V)=\infty$.   Then, the heat kernel associated to $L$ satisfies
\[p_t(x,y) \to 0, \ t\to \infty.\]
\end{coro}
\begin{proof}
Since $m(V) = \infty$ we have that $1\not\in \ell^2(V,m)$ and, in particular, that $1\not\in D(Q)$ and $0$ is not an eigenvalue. Now, if $E_0>0$, then, by the previous theorem, we get that $p_t(x,y) \to 0$ as $t\to\infty$. Otherwise, if $E_0=0$, then the theorem gives that $p_t(x,y) \to \Phi(x)\Phi(y)$ as $t\to\infty$. But $\Phi\equiv 0$ since $0$ is not an eigenvalue.
\end{proof}

 \textbf{Remark.} 
Let us emphasize that the assumptions in the preceding results  do not entail any form of compact resolvent or spectral gap. Accordingly, one can also not expect any form of exponential rate of convergence in the statements. For situations in which such assumptions and consequences hold we refer to the monograph \cite{EN}.

\section{Uniqueness of selfadjoint restrictions  of $\wl$} \label{Boundedness}
This section is concerned with  two situations in which there is only one selfadjoint restriction  of $\wl$. This complements the material of the previous   sections whose main thrust is the study of situations in which there are several selfadjoint restrictions of $\wl$.

\bigskip

We start with the following direct consequence of Proposition~\ref{maximaloperator}.

\begin{prop} \label{esa-prop} Let  $(b,c)$ be  a graph over the discrete measure space $(V,m)$ and $\wl = \wl_{(b,c,m)}$.  If $\mathrm{(FC)}$ holds and the restriction $L_c$ of $\wl$ to $C_c (V)$ is essentially selfadjoint with selfadjoint extension $L$, then $L$ is the only Laplacian associated to $(b,c)$.
\end{prop}

Now, we discuss two cases in which the proposition can be applied.  One case will involve an assumption on $m$ only  and the other case  will involve combined assumptions on $m$ and $(b,c)$.

\smallskip

\begin{coro}  Let $(V,m)$ be a discrete measure space with $\inf_{x\in V} m(x) >0$. Then, $\mathrm{(FC)}$ holds for any graph $(b,c)$ over $(V,m)$  and
  the restriction $L_c$ of $\wl$ to $C_c (V)$ is essentially selfadjoint with selfadjoint extension $L$.
\end{coro}
\begin{proof}  This follows from the previous proposition  by Theorem 6 of \cite{KL1} and its subsequent remark. Namely, as shown there, $\wl$ maps $C_c (V) $ to $\ell^2 (V,m)$ and the restriction of $\wl$ to $C_c (V)$ is essentially selfadjoint whenever $\inf_{x\in V}  m(x) >0$.
\end{proof}

\textbf{Remark.} As a consequence of the corollary the  main focus of the previous sections  concerns the case when $\inf_{x\in V} m(x) =0$.

\smallskip

Another instance of uniqueness of selfadjoint restrictions of $\wl$ is given if $\wl$ is a bounded operator. In this context we first discuss  an interesting feature of $\wl$ as an operator on $\ell^p (V,m)$: It is  either bounded for all $p\geq 1$ or for no such $p$.

\begin{theorem}\label{char-boundedness} Let $(V,m)$ be a discrete measure space, $(b,c)$ a graph over $(V,m)$ and $\wl=\wl_{(b,c,m)}$. Then, the following are equivalent:

\begin{itemize}

\item[(i)]  There exists a $C\geq 0$ with  $\frac{1}{m(x)} \left( \sum_y b(x,y) + c(x)\right) \leq C$ for all $x\in V$.

\item[(ii)]  $\wl$ gives a bounded operator on $\ell^p (V,m)$ for some $p\in [1,\infty]$.

\item[(iii)] $\wl$ gives a bounded operator on $\ell^p (V,m)$ for all $p\in [1,\infty]$.

\end{itemize}
\end{theorem}

\begin{proof}  We begin the proof with an auxiliary claim:

\smallskip

\textit{Claim.} If $\ell^p (V,m)$ is a subset of $\widetilde{F}$ and the restriction of $\wl$ to $\ell^p (V,m)$ is a bounded operator on $\ell^p (V,m)$ for some $p$ with $1\leq p \leq \infty$ and $q$ satisfies $1 = 1/p  + 1/q$,  then  $\ell^q (V,m)$ also belongs to $\widetilde {F}$ and the restriction of $\wl$ to $\ell^q (V,m)$ is a bounded operator as well.

\textit{Proof of the claim.} This is essentially a consequence of \eqref{PI} and  duality. We consider the cases $p = 1$ and $p > 1 $ separately.

To treat $p=1$ we note that $\ell^\infty (V)$ belongs to $\widetilde{F}$ anyway. From \eqref{PI} we then infer  for $u\in \ell^\infty (V)$ and $v\in C_c (V)$ the estimate
$$\left|\sum_x (\wl  u) (x) v (x) m(x)\right| = \left| \sum_x u(x) \wl v (x) m(x)\right| \leq  C \|u\|_\infty \|v\|_1,$$
where $\| \cdot\|_s$ denotes the $s$-norm on $\ell^s (V,m)$ and $C$ is a bound for the norm of $\wl$ as an operator from $\ell^1 (V,m)$ to $\ell^1 (V,m)$. As $C_c (V)$ is dense in $\ell^1 (V,m)$ we infer that $\wl$ is a bounded operator on $\ell^\infty (V)$ (with norm bounded by $C$ as well).

To treat the case $p>1$ we chose $u \in C_c (V) \subset \ell^q (V,m)$ and $v\in \ell^p (V,m)$. Then, \eqref{PI} gives the estimate
$$\left|\sum_x (\wl  u) (x) v (x) m(x)\right| = \left| \sum_x u(x) \wl v (x) m(x)\right| \leq  C \|u\|_q \|v\|_p,$$
where $C$ is a bound for the norm of $\wl$ as an operator from $\ell^p (V,m)$ to $\ell^p (V,m)$. This shows that $\wl u$ belongs to $\ell^q (V,m)$ and satisfies  $\|\wl u\|_q \leq C \|u\|_q$. Hence, the restriction of $\wl$ to $C_c (V)$ is a bounded operator (with respect to $\|\cdot\|_q$). It can then be (uniquely) extended to a bounded operator on $\ell^q (V,m)$ and this extension can easily be seen to be a restriction of $\wl$. This finishes the proof of the claim.

\medskip

Let us now turn to the actual proof.  Assume that (i) is satisfied. Then, we see that $\wl$ is a bounded operator on $\ell^\infty$ which gives (ii). The auxiliary claim now  gives  that $\wl$ is bounded on $\ell^1$. Applying the Riesz-Thorin interpolation  theorem, we get (iii).\\
Assume, conversely, that (ii) is fulfilled. Again by the  auxiliary claim, $\wl$ is also a bounded operator on $\ell^q$, where $1=\frac{1}{p} + \frac{1}{q}$. As $2$ must belong to the interval between $p$ and $q$ we can use interpolation once more, to  get that $\wl$ is bounded on $\ell^2$. Hence, for each $x\in V$, we have the existence of $C\ge0$ such that
\[\langle \wl \delta_x, \delta_x\rangle \leq C  m(x),\]
which gives (i).
\end{proof}

\textbf{Remark.} The equivalence of (i) and boundedness of $\wl$ on $\ell^2 (V,m)$ has been shown in \cite{KL2}. There it has also been shown that this implies boundedness of $\wl$ on all $\ell^p (V,m)$.

\begin{coro} Let $(V,m)$ be a discrete measure space, $(b,c)$ a graph over $(V,m)$ and $\wl=\wl_{(b,c,m)}$. If there exists a $C\geq 0$ with
$$\frac{1}{m(x)} \left( \sum_{y\in V} b(x,y) + c(x)\right) \leq C$$
for all $x\in V$, then there exists only one selfadjoint  Laplacian  associated to $(b,c)$. This Laplacian is  the restriction of $\wl$ to $\ell^2 (V,m)$.
\end{coro}
\begin{proof} By definition, any Laplacian associated to the graph is a restriction of $\wl$. As the restriction of $\wl$ to $\ell^2 (V,m)$ is bounded by the assumption and the previous theorem, the statement follows.
\end{proof}

\textbf{Remark.} Let us point out that the corollary gives examples  with uniqueness of the selfadjoint operator associated to $\wl$ even if $\inf_{x\in V} m(x) =0$.

\appendix
\section{A (Counter)example}\label{Counter}
In this section we briefly recall an example from Section $4$ of \cite{KL1}. This serves as a (counter) example in various situations discussed in the article.

\bigskip

We consider connected graphs $(b,c)$ over $(V,m)$ with $c\equiv 0$ and $b(x,y)\in \{0,1\}$ for all $x,y\in V$. The Cheeger constant $\alpha = \alpha (V,b,c)$ of such a graph is  defined by
$$\alpha :=\inf_{K\subseteq V : \sharp K <\infty} \frac{\qd(1_K, 1_K)}{\sharp K},$$
where $\sharp K$ denotes the cardinality of $K$.  The degree $D : V\longrightarrow [0,\infty)$ is defined by $D(x) = \sum_{y\in V} b(x,y)$.  Then, as discussed by Dodziuk-Kendall \cite{DK} (see \cite{Dod,Kel, KL2} as well) in the context of isoperimetric inequalities, the inequality
$$\frac{1}{2} \sum_{x,y\in V} b(x,y) (\varphi (x) - \varphi (y))^2  \geq \frac{\alpha^2}{2} \sum_{x\in V} D(x) \varphi (x)^2$$
holds for all real-valued $\varphi \in C_c (V)$.  Now, take an arbitrary graph with $\alpha >0$ and  $D(x)\leq C$ for all $x\in V$.  To be specific, one may take the regular tree with degree  $k\geq 3$.  Choose a measure $m$ on $V$ with $m(V) =1$. Then, the graph $(V,b,0,m)$ has the following features:

\begin{enumerate}
\item The graph $(V,b,0,m)$ is complete with respect to the metric $d$ defined in Section \ref{Forms}.
\item The constant function $1$ belongs to $D (\qn)$ but not to $D (\qd)$.
\end{enumerate}
Here, $(1)$ is clear as $D$ is uniformly bounded. Claim (2) is shown in \cite{KL1}. We include a short \textit{proof}: Let $1$ be  the constant function with value one.  As $m(V)<\infty$,  the function $1$ belongs to $\ell^2 (V,m)$.  Obviously, $\qn(1) := \qn (1,1)=0<\infty$. Thus, $1\in D(\qn)$. Now, fix $x_0\in V$ and let $\varphi_n$ be any sequence in $C_c (V)$ converging to $1$ in $\ell^2 (V,m)$. Then, $\varphi_n (x_0)$ converges to $1$. In particular,
$$\qn (\varphi_n) \geq  \frac{\alpha^2}{2} D (x_0) \varphi_n (x_0)^2\to \frac{\alpha^2}{2} D(x_0)>0, n\to \infty.$$
Thus, $\qn (\varphi_n)$ does not converge to $0   = \qn(1)$ and $1$ does not belong to $D(\qd)$.

\section{Dirichlet forms on real and complex Hilbert spaces} \label{real}
In this appendix we shortly discuss some basic and very well-known characterizations of Dirichlet forms.
 Such forms can be considered on both real and complex Hilbert spaces. Here, we show, in particular,  how to go from one situation to the other. We refer to  \cite{BH,Fuk,Ou} for further details. 

\bigskip

Let $(X,m)$ be a $\sigma$-finite measure space. Let $L^2 (X,m)$ be the Hilbert space of square integrable function on $X$ with inner product
$$\langle f, g\rangle = \int_X f (x) \overline{g(x)} d m(x).$$
As discussed in the main text, a  non-negative closed symmetric form $Q$ on $L^2 (X,m)$ with domain $D(Q)$ is called a Dirichlet form (on $L^2 (X,m)$)  if
$$Q ( C u, Cu )\leq Q (u,u)$$
for all $u\in L^2 (X,m)$ and every normal contraction $C$ on $\CC$.   Here, we set  $Q(v,v) = \infty$ if $v$ does not belong to the form domain.

In order to state the next theorem, we need some further notation: For a real valued $u$ we set
$$u_+ :=\max\{0,u\}\;\: \mbox{and}\;\: u \wedge 1 :=\min\{ 1, u\}.$$
 Then, the  following theorem holds.

\begin{theorem} Let $Q$ be a non-negative closed symmetric form  on $L^2 (X,m)$. Then, the following assertions are equivalent:

\begin{itemize}
\item[(i)] $Q$ is a Dirichlet form.

\item[(ii)] The semigroup $e^{-t L}$ is positivity preserving and contracting on $L^p (X,m)$ for all $1 \leq p \leq \infty$.
\item[(iii)]  $D(Q)$ is invariant under complex conjugation with $Q(u,v)$  real valued for all real valued $u,v$ in the domain of $Q$ and it holds that
$Q (u_+, u_+) \leq Q(u,u)$ for all real valued $u\in L^2 (X,m)$ as well as  $Q ( u \wedge 1, u \wedge 1) \leq Q (u,u)$ for all real valued $u\in L^2 (X,m)$ with $u \geq 0$.
\end{itemize}
\end{theorem}
\begin{proof}
The equivalence of (i) and (ii) follows  directly from  Theorems XIII.50 and XIII.51  in  Appendix 1 to Section XIII.12 of \cite{RS}.

\smallskip

Also, it is not hard to see that $D(Q)$ is invariant under complex conjugation with   $Q (u,v)$  real valued for all  real valued $u,v\in D(Q)$ if and only if $e^{-t L}$ maps real valued functions to real valued functions. Given this, the equivalence between (ii) and (iii) follows again  directly from  Theorems XIII.50 and XIII.51  in  Appendix 1 to Section XIII.12 of \cite{RS}.
\end{proof}

Let $L^2_\R (X,m)$ be the subspace of $L^2 (X,m)$ consisting of real-valued functions.
We then  call a non-negative closed symmetric form $Q$ on $L^2_\R (X,m)$ a Dirichlet form (on $L^2_\R (X,m)$) if
$$Q (C u, Cu) \leq Q (u,u)$$
for all $u\in L^2_\R (X,m)$ and every normal contraction $C$ on $\R$. (Here, we again set  $Q(v,v) = \infty$ if $v$ does not belong to the form domain.)
 Essentially from (iii) of the previous theorem we obtain the following corollary.

\begin{coro} (a) Let $Q$ be  a Dirichlet form on $L^2 (X,m)$. Then, the restriction of $Q$ to $L^2_\R (X,m)$ is a Dirichlet form as well.

(b) Let $Q_r$ be a Dirichlet form on $L^2_\R (X,m)$. Then, the form $Q$ with domain
 $$D(Q):= \{ u + i v : u, v\in D(Q_r)\}$$
 and
$$Q (u_1 + i v_1,  u_2 + i v_2) := Q_r(u_1, u_2) +  Q_r(v_1, v_2) + i (Q_r(v_1, u_2) -  Q_r(u_1, v_2))$$
is a Dirichlet form on $L^2 (X,m)$.
\end{coro}

\begin{proof} (a) By (iii) of the previous theorem, $Q(u,v)$ is real valued for all real valued $u,v\in D(Q)$. Now, (i) of the previous theorem, shows that the restriction of $Q$ to real valued functions is compatible with  taking normal contractions on $\R$.

\smallskip

(b) It is not hard to see that $Q$ is a symmetric closed non-negative form. Moreover, as $Q_r$ is compatible with contractions on $\R$ it is easy to see that (iii) of the previous theorem holds for $Q$. Thus, the previous theorem shows that $Q$ is a Dirichlet form.
\end{proof}

\medskip

\bigskip

\textbf{Acknowledgements.} Part of this work was done while the authors were visiting the workshop `Analysis on Graphs and its Applications Follow-up Meeting' at the Isaac Newton Institute. The authors would like to thank the organizers for the invitation and gratefully acknowledge  the stimulating atmosphere of the workshop.  In this context DL, SH and MK also gratefully acknowledge  partial support by the German Science Foundation (DFG).  RW acknowledges the financial support of FCT grant SFRH/BPD/45419/2008 and FCT project PTDC/MAT/101007/2008.

\end{document}